\documentclass[11pt,reqno]{amsart}
\usepackage{amsmath,amssymb,amsthm,amsfonts,enumerate,hyperref}
\usepackage{graphicx}
\usepackage{float}
\usepackage[margin=1.1in]{geometry}
\usepackage[dvipsnames]{xcolor}
\theoremstyle{plain}
\newtheorem{theorem2}{Theorem}
\newtheorem{theorem}{Theorem}
\newtheorem*{theorem*}{Theorem}

\newtheorem{lemma}[theorem2]{Lemma}

\theoremstyle{definition}

\theoremstyle{remark}

\numberwithin{equation}{section}

\begin{document}

\title{Bounding the invariant spectrum when the scalar curvature is non-negative.}
\author{Stuart J. Hall}
\address{School of Mathematics, Statistics and Physics, Herschel Building, Newcastle University, Newcastle upon Tyne, NE1 7RU, UK}
\email{stuart.hall@newcastle.ac.uk}
\author{Thomas Murphy}
\address{Department of Mathematics, California State University Fullerton, 800 N. State College Blvd., Fullerton, CA 92831, USA.}
\email{tmurphy@fullerton.edu}
\maketitle

\begin{abstract}
On compact Riemannian manifolds  with a large isometry group we investigate the invariant spectrum of the ordinary Laplacian. For either a toric K\"ahler metric, or a rotationally-symmetric metric on the sphere, we produce upper bounds for all eigenvalues of  the invariant spectrum assuming non-negative scalar curvature.   
\end{abstract}

\section{Introduction}
\noindent

Motivation for finding universal bounds on the spectrum of metrics on the sphere comes from a famous result due to Hersch \cite{Her}: any metric $g$ on the sphere $\mathbb{S}^{2}$, normalised to have volume $4\pi$, satisfies
$$\lambda_{1}(g)\leq 2.$$ In fact, Hersch also proved that $\lambda_1(g)=2$ if and only if the metric $g$ is isometric to the canonical round metric.\\
\\
If $G \leqslant \mathrm{Iso}(g)$ is a subgroup of the isometry group of a metric $g$,  the $G$-invariant spectrum is defined by restricting the Laplacian to $G$-invariant functions.   The $k^{\mathrm{th}}$ $G$-invariant eigenvalue of the ordinary Laplacian is denoted $\lambda^G_k(g)$.   In this setting, the analogue of Hersch's Theorem for metrics invariant under the standard action of $\mathbb{S}^1$ on $\mathbb{S}^2$ was studied by Abreu and Freitas. They proved that no counterpart for Hersch's Theorem holds for the $\mathbb{S}^{1}$-invariant spectrum. More specifically, for any $c\in (0,\infty)$, they were able to construct metrics $g$ with volume $4\pi$ and  first invariant eigenvalue $\lambda_{1}^{\mathbb{S}^1}(g)=c$.\\
\\
However, under additional assumptions on the geometry defined by $g$, they also proved the following theorem which is our primary inspiration in this paper.  A related set of results was also proved independently by Engman \cite{E1} \cite{E2}. To state the result, denote by $\xi_{k}$  the $\frac{1}{2}(k+1)^{\mathrm{th}}$ positive zero of the Bessel function $J_{0}$ if $k$ is odd, and the $\frac{1}{2}k^{\mathrm{th}}$ positive zero of $J_{0}'$ if $k$ is even. 

\begin{theorem*}[Abreu--Freitas, \cite{AF}]
Let $g$ be a metric on $\mathbb{S}^{2}$ which is invariant under the standard action of $\mathbb{S}^{1}$, with non-negative Gaussian curvature and  volume $4\pi$. Then 
$$\lambda_{k}^{\mathbb{S}^{1}}(g)< \frac{1}{2}\xi_{k}^{2}.$$
\end{theorem*}
\noindent
In the present note we are concerned with  two natural generalisations of this theorem to higher dimensions.  In both of them the  analogue of non-negative Gaussian curvature is to assume the metrics have  non-negative scalar curvature.  The two generalisations correspond to two different ways of viewing $\mathbb{S}^{1}$-invariant metrics on $\mathbb{S}^{2}$. The first is to realize $\mathbb{S}^{2}$ as a complex manifold (with complex dimension one), $g$ as a K\"ahler metric, and to view $\mathbb{S}^{1}$ as a real, one-dimensional torus $\mathbb{T}$. In higher dimensions, a well-developed theory of K\"ahler metrics   which are invariant under the Hamiltonian action of a real $n$-torus $\mathbb{T}^{n}$ exists, namely \textit{toric K\"ahler manifolds}.  In this setting we prove the following result.  
\begin{theorem} \label{TKthm}
Let $(M,\omega)$ be a compact toric K\"ahler manifold with non-negative scalar curvature. Then there exist constants $0<C_{k}([\omega])$, depending only upon the cohomology class $[\omega] \in H^{2}(M;\mathbb{R})$, such that 
$$\lambda^{\mathbb{T}^{n}}_{k}\leq C_{k}.$$
\end{theorem}
\noindent
The $k=1$ bound in Theorem \ref{TKthm} was established by the authors with a slightly different proof  in \cite{HMMRL}; here  a refined method yields all higher eigenvalue bounds.\\ 
\\
The second viewpoint we take is to identify  $\mathbb{S}^{1}\cong SO(2)$ and  $SO(2) \leqslant SO(3)$ via the embedding
$$ g\in SO(2) \rightarrow \left(\begin{array}{cc} 1 & 0\\ 0 & g  \end{array} \right) \in SO(3).$$
The action is  inherited from the natural action of $SO(3)$ on $\mathbb{R}^{3}$ by regarding the sphere as a hypersurface $\mathbb{S}^{2}\subset\mathbb{R}^{3}$. In higher dimensions this leads to the action of the group $SO(n)$ on the $n$-dimensional sphere $\mathbb{S}^{n}\subset\mathbb{R}^{n+1}$.  This example is part of another well-established theory, that of \textit{cohomogeneity one metrics}. We prove the following result. 
\begin{theorem} \label{C1thm}
There exist constants $D_{k}>0$ such that for any unit-volume, $SO(n)$-invariant metric $g$ on $\mathbb{S}^{n}$ with non-negative scalar curvature, we have the bounds
$$\lambda_{k}^{SO(n)}(g)\leq D_{k}.$$  
\end{theorem}
\noindent
We should say immediately that Theorem \ref{C1thm} is not really new and is essentially contained in the work of Colbois, Dryden and El Soufi (c.f. Theorem 1.7 in \cite{CDE}). However, in our proof one can see very directly how the assumption on the non-negativity of the scalar curvature is used.  

\section{Background}
\subsection{Previous results in the literature}
 In all the results we mention here, $M$ is assumed to be compact and without boundary.  Korevaar \cite{Kor}, building on work by Yang and Yau \cite{YY}, gave a far-reaching generalisation of the Hersch result to higher eigenvalues proving that, for any oriented Riemannian surface $(S,g)$ of genus $\gamma$,  there exists a universal constant $C>0$ such that
$$\lambda_{k}(g)\leq \frac{C(\gamma+1)k}{\mathrm{Vol}(S,g)},$$
where $\lambda _{k}(g)$ is the $k^{\mathrm{th}}$ eigenvalue in the spectrum of the Laplacian.\\
\\
In the case that the manifold has dimension 3 or higher, Colbois and Dodziuk \cite{CD} proved that there is a unit volume metric $g$ with arbitrarily large $\lambda_{1}(g)$; hence there is no hope of a general analogue of  the Hersch and Korevaar results for arbitrary manifolds.  However, in the case that the manifold $M$ is a compact, projective, complex manifold, Bourguignon, Li and Yau \cite{BLY} demonstrated an upper bound for $\lambda_{1}(g)$ depending upon some of the data associated with embedding $M$ into complex projective space $\mathbb{CP}^{N}$.  This result has been recently extended to higher eigenvalues by Kokarev \cite{Kok}.\\
\\
The study of the invariant spectrum for the (effective and non-transitive) action for a general compact Lie group $G$ was taken up by Colbois, Dryden and El Soufi in \cite{CDE}. They proved that, if the dimension of $G$ is at least one, then 
$$\sup\left\{\lambda_{1}^{G}(g)\mathrm{Vol}(g)^{2/n}\right\} =\infty,$$
where the supremum is taken over all $G$-invariant metrics conformal to a fixed reference metric $g_{0}$.

Finally, in the case of $n$-complex-dimensional toric K\"ahler manifolds, we mention that Legendre and Sena-Dias \cite{LSD} have generalised the results of Abreu and Freitas to show that the first $\mathbb{T}^{n}$-invariant eigenvalue, $\lambda_{1}^{\mathbb{T}^{n}}(g)$, can take any value in $(0,\infty)$  when the metric $g$ is allowed to vary over the set of toric Kahler metrics in a fixed cohomology class.  The Colbois--Dryden--El Soufi and Legendre--Sena-Dias results make no assumption about the curvature of the metrics and so do not contradict Theorems A and B. 

\subsection{Toric K\"ahler manifolds}
In order to prove Theorem \ref{TKthm} we recall some facts about toric K\"ahler metrics.  We cannot hope  to give a comprehensive introduction to this subject but refer readers to the works of Abreu \cite{A} and Donaldson \cite{D} which both contain detailed discussions of the theory.  In a nutshell, associated to any Hamiltonian action of an $n$-torus $\mathbb{T}^{n}$ on an $n$-complex-dimensional K\"ahler manifold $(M,\omega,J)$ is a convex polytope $P\subset \mathbb{R}^{n}$.  The polytope $P$, often referred to in the literature as the {\it moment polytope}, depends only upon the cohomology class of the K\"ahler metric ${[\omega]\in H^{2}(M;\mathbb{R})}$. There is also a dense open set $M^{\circ}\subset M$ such that $M^{\circ} \cong P^{\circ}\times (0,2\pi)^{n}$,  where $P^{\circ}$ is the interior of $P$.  In the natural coordinates, the metric ${g(\cdot,\cdot) := \omega(J\cdot,\cdot)}$ is given by
$$ g = u_{ij}dx_{i}dx_{j}+u^{ij} d\theta_{i}d\theta_{j},$$
where $u_{ij}$ is the Euclidean Hessian of a convex function $u:P^{\circ}\rightarrow \mathbb{R}$ and $u^{ij}$ is the inverse of this matrix. The function $u$ is referred to as the \textit{symplectic potential} of the metric $g$. Clearly, in these coordinates, the volume form of the metric is the standard Euclidean form 
$$dV = dx_{1}\wedge dx_{2} \wedge ...\wedge dx_{n}\wedge d\theta_{1}\wedge d\theta_{2} \wedge...\wedge d\theta_{n}.$$ Crucial to our proof is the formula for the scalar curvature in the $(x_{i},\theta_{j})$ coordinates which was first given by Abreu \cite{A2} (often called the Abreu equation in the literature):
$$\mathrm{Scal}(g) = -\sum_{i,j} \frac{\partial^{2} u^{ij}}{\partial x_i\partial x_j} .$$
This expression yields the following useful integration-by-parts formula due to Donaldson \cite{D2}.
\begin{lemma}[Donaldson's Integration-by-parts formula, Lemma 3.3.5 in \cite{D2}]
Let ${u:P^{\circ}\rightarrow\mathbb{R}}$ be a symplectic potential of a toric K\"ahler metric $g$ and let $F\in C^{\infty}(P)$. Then
\begin{equation}\label{DIBP}
\int_{P}u^{ij}F_{ij}d\mu = \int_{\partial P}2Fd\sigma-\int_{P}\mathrm{Scal}(g)Fd\mu,
\end{equation}
where $\mu$ is the usual Lesbegue measure on $\mathbb{R}^{n}$ and $\sigma$ is the integral Lesbegue measure on the boundary $\partial P$. 
\end{lemma}
Here the subscripts on $F$ denote partial differentiation.  Equation (\ref{DIBP})  differs slightly from the expression in Lemma 3.3.5 in \cite{D2} since we adopt a different convention to describe  the singular part of the symplectic potential. We refer the reader to \cite{HMMRL} for futher discussion of our conventions and the precise definition of the integral Lesbegue measure $\sigma$.

\section{Toric K\"ahler eigenvalue bounds}
We now give the proof of Theorem \ref{TKthm}.  Central to the proof is the fact that $\sigma$ does not depend upon the metric $g$. Rather it depends only on the cohomology class of the associated K\"ahler form $\omega$. 
\begin{proof}
Fix the moment polytope $P \subset \mathbb {R}^{n}$ and pick a coordinate $x \in [x_{\min},x_{\max}]$. Throughout the proof we shall abuse notation by denoting by $f$ functions in this single variable ${f:[x_{\min},x_{\max}]\rightarrow \mathbb{R}}$ and the extension to the whole polytope $P$ given by $(f\circ \pi)$ where $\pi:P\rightarrow [x_{\min},x_{\max}]$ is the projection to this coordinate.\\
\\
Let $V_{k} \subset W^{1,2}$ be the subspace given by
$$V_{k} = \mathrm{span}\{1,x, x^{2},\dots, x^{k}\}.$$ 
Hence the min-max characterisation of $\lambda_{k}^{\mathbb{T}^{n}}$ yields 
$$
\lambda_{k}^{\mathbb{T}^{n}} \leq \sup_{\varphi \in V_{k}, \varphi \neq 0}\left\{ \frac{\int_{P}\|\nabla \varphi \|^{2}d \mu}{\int_{P}\varphi^{2}d\mu}\right\}.
$$
We denote by $\varrho:[x_{\min},x_{\max}]\rightarrow \mathbb{R}$ the unique function satisfying
$$\varrho_{xx} = (\varphi_{x})^{2} \qquad \mathrm{and} \qquad\varrho(x^{\ast})=\varrho_{x}(x^{\ast})=0,$$ 
where $x^{\ast} = \dfrac{x_{\min}+x_{\max}}{2}$.
As $\varrho$ is manifestly convex, the point $x^{\ast}$ is a global minimum and so $\varrho$ is non-negative on $[x_{\min},x_{\max}]$. 
\\

The Raleigh quotient thus furnishes us with a map ${R:V_{k}\rightarrow \mathbb{R}}$
$$
R(\varphi) = \frac{\int_{P}\|\nabla \varphi \|^{2}d \mu}{\int_{P}\varphi^{2}d\mu}.
$$
The numerator can be evaluated using Equation (\ref{DIBP})
$$
\int_{P}\|\nabla \varphi \|^{2}d \mu =\int_{P}u^{xx}(\varphi_{x})^{2} d\mu = \int_{P}u^{xx}\varrho_{xx}d\mu = \int_{\partial P}2\varrho d\sigma - \int_{P}\mathrm{Scal}(g)\varrho d\mu.
$$
As the final term in the previous equation is non-negative, we obtain the bound
$$
\int_{P}\|\nabla \varphi \|^{2}d \mu \leq \int_{\partial P}2\varrho d\sigma. 
$$
Hence
$$R(\varphi) \leq \frac{\int_{\partial P}2\varrho d \sigma}{\int_{P}\varphi^{2}d\mu},
$$
where the bound is independent of the metric. The function ${B: V_{k}\backslash \lbrace 0 \rbrace \rightarrow \mathbb{R}}$ given by
$$B(\varphi) = \frac{\int_{\partial P}2\varrho d \sigma}{\int_{P}\varphi^{2}d\mu}, $$
is continuous.  We demonstrate this by considering the numerator and denominator separately. We endow the space $V_k$ with the standard Euclidean topology.  The map ${\mathcal{F}_{1}:V_{k}\rightarrow V_{2(k-1)}}$ given by $\mathcal{F}_{1}(\varphi) = (\varphi_{x})^{2}$ is continuous (this could be seen by writing it out in coordinates).  The map ${\mathcal{F}_{2}:V_{2(k-1)}\rightarrow V_{2k}}$ given by $\mathcal{F}_{2}(f)= \varrho$, where $\varrho$ solves the equation
$$\varrho_{xx} = f \qquad \mathrm{and} \qquad\varrho(x^{\ast})=\varrho_{x}(x^{\ast})=0,$$
is linear and so continuous (note the domain is finite dimensional). The map ${\mathcal{F}_{3}:V_{2k}\rightarrow \mathbb{R}}$ given by
$$ \mathcal{F}_{3}(f) = \int_{\partial P}2fd\sigma,$$
is also linear and so continuous. The numerator of the function $B$ is given by the composition ${\mathcal{F}_{3}\circ\mathcal{F}_{2}\circ \mathcal{F}_{1}(\varphi)}$ and thus is continuous. A very similar argument yields the continuity of the denominator as a map from ${V_k\backslash \lbrace 0 \rbrace }$ to $\mathbb{R}$. Thus we obtain a continuous function $B: V_k\backslash \lbrace 0 \rbrace \rightarrow \mathbb{R}$.\\
\\
If we scale ${\varphi \rightarrow \alpha \varphi}$ for $\alpha \in \mathbb{R}$, then $\varrho \rightarrow \alpha^{2}\varrho$. Hence $B$ is determined by the values it takes on the functions $\varphi \in \mathbb{S}^{k}\subset V_{k}$ (where we take the obvious coordinates on $V_{k}$). Restricting $B$ to the sphere yields a continuous function, so   the constant $C_{k}$ can be determined by taking the maximum of $B$ as $\varphi$ varies over the sphere $\mathbb{S}^{k}$.  

\end{proof}

\section{Cohomogeneity one metrics on the sphere}

Let $g$ be a unit-volume metric on $\mathbb{S}^{n}$ given by $$g=dt^{2}+\rho^{2}(t)d\mathbb{S}^{2}_{n-1},$$ 
where $d\mathbb{S}^{2}_{n-1}$ is  the  unit-volume round  metric on the sphere $\mathbb{S}^{n-1}$ and $\rho:[0,b]\rightarrow \mathbb{R}$ is such that
$$
\rho(0)=\rho(b)=0, \qquad \dot{\rho}(0)=-\dot{\rho}(b)=1 \qquad \mathrm{and } \quad  \rho>0 \quad \mathrm{ on } \quad (0,b),
$$
where dots over a function denotes derivatives with respect to $t$.  In these coordinates the scalar curvature is given by
\begin{equation}\label{scaleqn}
\mathrm{Scal}(t) = -2(n-1)\frac{\ddot{\rho}}{\rho}+(n-1)(n-2)\frac{1-\dot{\rho}^{2}}{\rho^{2}}.
\end{equation}
This immediately yields the following result.
\begin{lemma}\label{rholem}
Let $\rho$ be as above and suppose the metric $g$ has non-negative scalar curvature. Then $$|\dot{\rho}|\leq 1.$$
\end{lemma}
\begin{proof}
Suppose the bound does not hold, then the boundary conditions imply that $\dot{\rho}$ has a global maximum or minimum, $t^{\ast}$ say, where $t^{\ast}\in(0,b)$. From Equation (\ref{scaleqn}), at such a point the scalar curvature is given by
$$
\mathrm{Scal}(t^{\ast}) = (n-1)(n-2)\frac{1-\dot{\rho}^{2}(t^{\ast})}{\rho^{2}(t^{\ast})}.
$$
If $|\dot{\rho}(t^{\ast})|>1$ there is now an immediate contradiction as the scalar curvature would be strictly negative at this point. 
\end{proof}
We will now use a new coordinate $s$ defined by a map $s(t):[0,b]\rightarrow [0,1]$ given by
$$s(t):=\int_{0}^{t} \rho^{n-1}(\tau) d\tau.$$
As the integrand is positive, $s(t)$ is invertible and we denote its inverse as $t(s)$.   
In these $s$--coordinates the metric becomes
\begin{equation}\label{gc1met}
g = \frac{ds^{2}}{\Phi^{2n-2}}+\Phi^2d\mathbb{S}_{n-1}^{2},
\end{equation}
where $\Phi(s) := \rho(t(s))$. 

Straightforward calculation, with the convention that derivatives with respect to $s$ are denoted with a  prime,  yields
$$
\dot{\rho}(t(s)) = \Phi'(s)(\Phi(s))^{n-1} = \frac{1}{n}(\Phi^{n})',
$$
and so
$$ 
\Phi^{n}(0) = \Phi^{n}(1)=0, \qquad (\Phi^{n})'(0) = -(\Phi^{n})'(1) = n.
$$
Combining this with Lemma \ref{rholem} gives
\begin{lemma}\label{Philem}
	Let $\Phi:[0,1]\rightarrow\mathbb{R}$ be as  above. Then 
	$$\Phi(s)\leq \Phi_{\max}(s),$$
	where
	$$\Phi_{\max}(s) = \left\{\begin{array}{c} ^{n}\sqrt{ns} \quad \mathrm{if} \quad 0 \leq s \leq \dfrac{1}{2}, \\
	^{n}\sqrt{\dfrac{n}{2}-n(s-\dfrac{1}{2})} \quad \mathrm{if} \quad \dfrac{1}{2} \leq s \leq 1.
	\end{array}\right.$$	 
\end{lemma}
We can now prove Theorem \ref{C1thm}. It essentially follows from the fact that $\Phi_{max}$ is independent of metric. 
\begin{proof}
Working in the `$s$' coordinate, we consider the subspace $V_{k}\subset W^{1,2}$ given by 
$$V_{k}: =\mathrm{span}\left\{1,s,s^{2},...,s^{k}\right\}.$$
We take $\varphi\in\mathbb{S}^{k}\subset V_{k}$ (where we use the obvious coordinates on $V_{k}$ to define the sphere) and consider
$$R(\varphi):=\frac{\int_{\mathbb{S}^{n}}|\nabla \varphi|^{2} dV_{g}}{\int_{\mathbb{S}^{n}} \varphi^{2} dV_{g}}.$$ 
Using the specific form of the metric $g$ in Equation (\ref{gc1met}) we obtain
$$R(\varphi) = \frac{\int_{0}^{1}\Phi^{2n-2}(\varphi')^{2} ds}{\int_{0}^{1}\varphi^{2} ds} \leq \frac{\int_{0}^{1}\Phi_{\mathrm{max}}^{2n-2}(\varphi')^{2} ds}{\int_{0}^{1}\varphi^{2} ds}$$
where the last inequality uses Lemma \ref{Philem}.  Thus the Raleigh quotient is clearly bounded above by a quantity that does not depend upon the metric $g$.  The result follows from the min-max characterisation of the $k^{th}$ eigenvalue.
\end{proof}

\end{document}